\documentclass[11pt,a4paper]{amsart}

\usepackage[myheadings]{fullpage}
\usepackage{amsmath,amssymb,amsthm}
\usepackage{amsfonts,enumerate,mathtools}
\usepackage{hyperref}
\hypersetup{
    colorlinks=true,   
    linkcolor=blue,
    citecolor=blue,
}

\newtheorem{cor*}{Corollary}
\newtheorem{prop*}{Proposition}
\newtheorem{thm*}{Theorem}

\newtheorem{theorem}{Theorem}[section]
\newtheorem{cor}{Corollary}[theorem]
\newtheorem{prop}[theorem]{Proposition}
\newtheorem{lem}[theorem]{Lemma}

\newcommand{\s}{\mathbb{S}}
\newcommand{\Z}{\mathbb{Z}}
\DeclareMathOperator{\pmap}{PMod}
\DeclareMathOperator{\lmap}{LMod}
\DeclareMathOperator{\smap}{SMod}
\DeclareMathOperator{\map}{Mod}
\everymath{\displaystyle}

\begin{document}

\title[Liftable mapping class groups of certain branched covers of torus]{Liftable mapping class groups of certain branched covers of torus}

\author{Pankaj Kapari}
\address{School of Mathematics,
Tata Institute of Fundamental Research,
Dr. Homi Bhabha Road, Navy Nagar, Colaba,
Mumbai, 400005, Maharashtra India}

\email{pankajkapri02@gmail.com}
\urladdr{https://sites.google.com/view/pankajkapdi/}

\subjclass[2020]{Primary 57K20, Secondary 57M60}

\keywords{hyperbolic surfaces, branched covers, liftable mapping class groups, Birman-Hilden theory}

\begin{abstract}
Let $S_{g,n}$ be a closed oriented hyperbolic surface of genus $g$ with $n$ marked points, with the understanding that $S_{g,0}=S_g$. Let $\mathrm{Mod}(S_{h,n})$ be the mapping class group of $S_{h,n}$ and $\mathrm{LMod}_p(S_{h,n})$ be the liftable mapping class group associated to a cover $p:S_g\to S_{h,n}$. For the cover $p_k:S_k\to S_{1,2}$, Ghaswala, in his PhD thesis, derived a finite presentation for $\mathrm{LMod}_{p_k}(S_{1,2})$ when $k=2,3,4$ and a finite generating set when $k=5,6$ using the Reidemeister-Schreier rewriting process. In this paper, we derive a finite generating set for $\mathrm{LMod}_{p_k}(S_{1,2})$ for all $k\geq 2$. In the process, we also prove that the kernel of the homology representation $\Psi:\mathrm{Mod}(S_{1,2})\to \mathrm{GL}_3(\Z)$ is normally generated by a Dehn twist about a separating simple closed curve, and it is free with a countable basis. We also provide an explicit countable basis for $\ker\Psi$ consisting of separating Dehn twists. As an application of Birman-Hilden theory, we provide a finite generating set for the normalizer of the Deck group of $p_k$ in $\mathrm{Mod}(S_k)$ when $k=2,3$. We conclude the paper by proving that $\mathrm{LMod}_{p_k}(S_{1,2})$ is maximal in $\mathrm{Mod}(S_{1,2})$ if and only if $k$ is prime. 
\end{abstract}

\maketitle

\section{Introduction}
\label{sec:into}
Let $S_{g,n}$ denote the closed oriented hyperbolic surface of genus $g\geq 0$ and $n\geq 0$ marked points. The surface $S_{g,0}$ will be denoted by $S_g$. The \textit{mapping class group} of $S_{g,n}$, denoted by $\map(S_{g,n})$, is the group of isotopy classes of orientation-preserving self-homeomorphisms of $S_{g,n}$ that preserve the set of marked points. The elements of $\map(S_{g,n})$ are called \textit{mapping classes}. All covers will be assumed to be finite-sheeted, regular, and branched over hyperbolic surfaces unless stated otherwise. For a cover $p:S_g\to S_{h,n}$, the \textit{liftable mapping class group} of $p$, denoted by $\lmap_p(S_{h,n})$, is the subgroup of $\map(S_{h,n})$ consisting of mapping classes represented by homeomorphisms that lift under the cover $p$. For the cover $p_k:S_k\to S_{1,2}$ (see Figure~\ref{fig:s3_s12_cover} for $k=2,3$), we study $\lmap_{p_k}(S_{1,2})$ and provide a finite generating set for $\lmap_{p_k}(S_{1,2})$ for all $k\geq 2$. We now provide an account of past works and recent developments in the understanding of liftable mapping class groups of covers.
\begin{figure}
\centering
\includegraphics[scale=1.2]{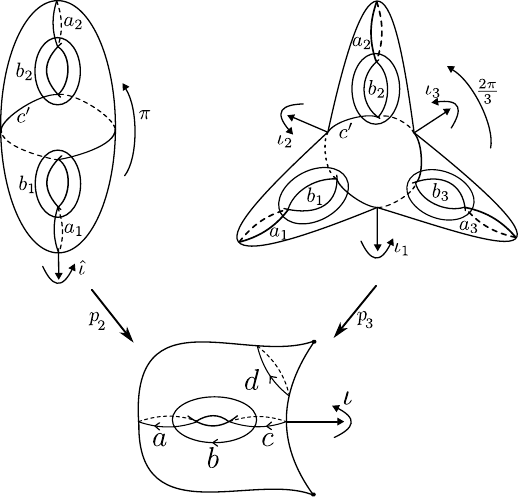}
\caption{For $k=2,3$, a $k$-sheeted cyclic cover $p_k:S_k\to S_{1,2}$.}
\label{fig:s3_s12_cover}
\end{figure}

The \textit{symmetric mapping class group}, denoted by $\smap_p(S_{g})$, is the subgroup of $\map(S_g)$ consisting of mapping classes represented by fiber-preserving homeomorphisms. The cover $p$ is said to have the \textit{Birman-Hilden property} if two isotopic fiber-preserving homeomorphisms are isotopic through fiber-preserving homeomorphisms. In 1971, Birman-Hilden~\cite[Theorem 7]{birman71} established this property for the cover $p:S_2\to S_{0,6}$ induced by the hyperelliptic involution. Consequently, they derived~\cite[Theorem 8]{birman71} the first known finite presentation of $\map(S_2)$ by proving that $\smap_p(S_2)=\map(S_2)$, $\lmap_p(S_{0,6})=\map(S_{0,6})$, and lifting a finite presentation of $\lmap_p(S_{0,6})$ under the map $\smap_p(S_2)\to \lmap_p(S_{0,6})$. Later on, Birman-Hilden~\cite[Theorem 4-5]{birman73} and Harvey-Maclachlan~\cite[Corollary 12]{harvey75} established this property for all covers. If $p$ has the Birman-Hilden property, then there is a short exact sequence~\cite[Theorem 10]{harvey75}
\[
1 \to \mathrm{Deck}(p)\to \smap_p(S_g)\to \lmap_p(S_{h,n})\to 1,
\]
where $\mathrm{Deck}(p)$ is the group of Deck transformations of $p$. Furthermore, $\smap_p(S_g)$ is the normalizer $N_{\map(S_g)}(\mathrm{Deck}(p))$ of $\mathrm{Deck}(p)$ in $\map(S_g)$~\cite[Theorem 4]{birman73}. For a finite order mapping class $F\in \map(S_g)$, Birman-Hilden~\cite[Section 5]{birman73} gave a finite presentation for $N_{\map(S_g)}(F)$ assuming that that $\langle F \rangle$ induces a cover of the sphere under which every mapping class lifts. For a survey of Birman-Hilden property and liftable mapping class groups, we refer the interested readers to~\cite{margalit21}.

In 2017, Ghaswala-Winarski~\cite[Theorem 1.3]{winarski17} classified all covers $S_g\to S_{0,n}$ under which every homeomorphism lifts. In a subsequent work~\cite[Theorem 5.7]{ghaswala17}, they derived a finite presentation for $\lmap_p(S_{0,n})$ associated to a balanced superelliptic cover $p:S_g\to S_{0,n}$ (a hyperelliptic cover is a particular case of this cover). Later on, Hirose-Omori~\cite[Theorem 5.3]{hirose_liftable} derived a finite presentation for the normalizer of the balanced superelliptic map by lifting a presentation of the liftable mapping class group. Abelian branched covers of a sphere and a real projective plane under which every homeomorphism lifts have also been classified (see~\cite{abelian_all_lift1,abelian_all_lift2, abelian_all_lift_rp2}). 

For covers $p:S_g\to S_{0,n}$, in a recent work~\cite[Theorem 3.3 and Remark 3.4]{kapdi_liftable1}, using an algebraic characterization of liftability due to Broughton~\cite[Theorem 3.2]{broughton92}, the authors gave a method to derive finite generating sets for $\lmap_p(S_{0,n})$. Consequently, they derived~\cite[Corollary 4.6.1]{kapdi_liftable1} a finite presentation for $N_{\map(S_g)}(F)$ and the centralizer $C_{\map(S_g)}(F)$ of a finite order reducible mapping class $F\in \map(S_g)$ of the highest order $2g+2$ (see~\cite[Theorem 4.1]{2g+2} for the upper bound on the order of reducible finite order mapping classes). For the unbranched cover $p:S_{k(g-1)+1}\to S_g$, using the symplectic criterion~\cite[Theorem 2.2]{dhanwani_liftable1}, a finite generating set for $\lmap_p(S_g)$ has been obtained in~\cite[Theorem 3.9]{dhanwani_liftable2} for $g\geq 3$ and $k\geq 2$, and in~\cite[Theorem 3.10]{kapdi_liftable2} for $g=2$ and $k$ prime.   

For the cover $p_k:S_k\to S_{1,2}$, Ghaswala~\cite[Section 5.1.3]{ghaswala_thesis}, in his PhD thesis, derived a finite presentation for $\lmap_{p_k}(S_{1,2})$ when $k=2,3,4$ and a finite generating set when $k=5,6$ using the Reidemeister-Schreier rewriting process on SageMath. We aim to derive a finite generating set for $\lmap_{p_k}(S_{1,2})$ for all $k\geq 2$. To achieve our aim, we use the homology representation $\Psi:\map(S_{1,2})\to \mathrm{GL}_3(\Z)$ of $\map(S_{1,2})$ induced by the action of $\map(S_{1,2})$ on $H_1(S_{1,2};\Z)$.

We now describe the results obtained in this paper while simultaneously providing an outline of the paper. In Section~\ref{sec:prelim}, we prove the following result (see Proposition~\ref{prop:genset_image}) which provide a finite generating set for $\Psi(\lmap_{p_k}(S_{1,2}))$. For a simple closed curve $c$, a left-handed Dehn twist about $c$ will be denoted by $T_c$.

\begin{prop*}
For curves as in Figure~\ref{fig:s3_s12_cover}, we have $$\Psi(\lmap_{p_k}(S_{1,2}))=\langle \Psi(T_a),\Psi(T_b),\Psi(T_c^k),\Psi(\iota)\rangle.$$
\end{prop*}

The kernel of the symplectic representation $\map(S_2)\to \mathrm{SP}_4(\Z)$, denoted by $\mathcal{I}(S_2)$, is called the Torelli group of $S_2$. It is known that $\mathcal{I}(S_2)$ is normally generated by a Dehn twist about a separating simple closed curve~\cite[Theorem 2]{birman_torelli}. Furthermore, $\mathcal{I}(S_2)$ is not finitely generated~\cite{mccullough86}, and it is free with a countable basis~\cite[Proposition 4]{mess_torelli}. But, an explicit free basis of $\mathcal{I}(S_2)$ consisting of separating Dehn twists is not known. In this context, in Section~\ref{sec:genset_kerpsi}, we prove that $\ker \Psi$ is free with a countable basis, and it is normally generated by a separating Dehn twist. We also provide an explicit free basis for $\ker\Psi$ consisting of countably many separating Dehn twists (see Theorem~\ref{thm:genset_ker}).

\begin{thm*}
The kernel of $\Psi$ is isomorphic to the commutator subgroup $[F_2,F_2]$ of a free group $F_2$ of rank $2$ and it is normally generated by a separating Dehn twist. Furthermore, for curves as in Figure~\ref{fig:s3_s12_cover}, we have
\[
\ker\Psi= \langle T_b^{-n}T_cT_a^{-1}T_b^nT_a^m(T_bT_c)^6T_a^{-m}T_b^{-n}T_aT_c^{-1}T_b^n \mid m,n \in \Z \rangle.
\]
\end{thm*}

In Section~\ref{sec:genset_lmod}, we obtain a finite generating set for $\lmap_{p_k}(S_{1,2})$ by combining generating sets of $\Psi(\lmap_{p_k}(S_{1,2}))$ and $\ker\Psi$ (see Theorem~\ref{thm:genset_lmod}).

\begin{thm*}
For $k\geq 2$, consider the $k$-sheeted cyclic cover $p_k:S_k \to S_{1,2}$. For curves as in Figure~\ref{fig:s3_s12_cover}, we have
\[
\lmap_{p_k}(S_{1,2})=\langle T_a,T_b,T_c^k,\iota, (T_bT_c)^6, T_c^jT_b^iT_a(T_bT_c)^6T_a^{-1}T_b^{-i}T_c^{-j}\mid 1\leq i,j <k \rangle,
\]
where $\iota $ is the hyperelliptic involution.
\end{thm*}

When $k=2,3$, we simplify the above generating set (see Proposition~\ref{prop:genset_lmod_23}) to recover the generating sets obtained by Ghaswala~\cite[Section 5.1.3]{ghaswala_thesis} using the Reidemeister-Schreier rewriting process.

\begin{prop*}
For $k=2,3$, we have $\lmap_{p_k}(S_{1,2})=\langle T_a,T_b,T_c^k,\iota\rangle$. 
\end{prop*}

Let $F_k\in \map(S_k)$ be the mapping class representing the $2\pi/k$-rotation of $S_k$ inducing the cover $p_k:S_k\to S_{1,2}$. As a corollary, we provide a finite generating set of $N_{\map(S_k)}(F_k)$ for $k=2,3$ (see Corollary~\ref{cor:genset_normalizer}) by lifting generating sets of $\lmap_{p_k}(S_{1,2})$ under the map $N_{\map(S_k)}(F_k)\to \lmap_{p_k}(S_{1,2})$.

\begin{cor*}
For curves shown in Figure~\ref{fig:s3_s12_cover}, we have
\[
N_{\map(S_2)}(F_2)=\langle F_2,T_{a_1}T_{a_2}, T_{b_1}T_{b_2},T_{c'},\hat{\iota} \rangle,
\]
where $\hat{\iota}\in \map(S_2)$ is the hyperelliptic involution, and
\[
N_{\map(S_3)}(F_3)=\langle F_3,T_{a_1}T_{a_2}T_{a_3}, T_{b_1}T_{b_2}T_{a_3},T_{c'},\iota_1\iota_2\iota_3 \rangle,
\]
where $\iota_1,\iota_2,\iota_3\in \map(S_3)$ are involutions represented by $\pi$-rotations as shown in Figure~\ref{fig:s3_s12_cover}.
\end{cor*}

For $p:S_{k(g-1)+1}\to S_g$, in~\cite[Theorem 2.4]{dhanwani_liftable2}, it was proved that $\lmap_p(S_g)$ is maximal in $\map(S_g)$ if and only if $k$ is prime. In Section~\ref{sec:maximal}, we conclude the paper with a similar result (see Theorem~\ref{thm:maximal}).

\begin{thm*}
Consider the cover $p_k:S_k\to S_{1,2}$. Then $\lmap_{p_k}(S_{1,2})$ is maximal in $\map(S_{1,2})$ if and only if $k$ is prime.
\end{thm*}

\section{Prelimnaries}
\label{sec:prelim}
Let $S_{h,n}^{\circ}$ be the surface obtained by removing marked points from $S_{h,n}$. For an abelian finite group $G$, from covering space theory, a surjective map $H_1(S_{h,n}^{\circ};\Z)\to G$ defines a finite-sheeted regular branched cover $S_g\to S_{h,n}$ for some $g$ with Deck transformation group isomorphic to $G$. In this paper, we consider the $k$-sheeted cyclic cover $p_k:S_k \to S_{1,2}$ induced by $2\pi/k$-rotation of a flower shaped $S_k$ with $k$-petals (see Figure~\ref{fig:s3_s12_cover} for $k=2,3$). Let $c_1$, $c_2$, and $c_3$ be the homology classes of simple closed curves $a$, $b$, and $d$, respectively, on $S_{1,2}$ as shown in Figure~\ref{fig:s3_s12_cover}. Then $H_1(S_{1,2}^{\circ};\Z)=\langle c_1,c_2,c_3 \rangle \cong \Z^3$. For $k\geq 2$, the map $\eta_k:H_1(S_{1,2}^{\circ};\Z)\to \Z_k$ given as $\eta_k: n_1c_1+n_2c_2+n_3c_3\mapsto n_3 \pmod k$ defines the cyclic $k$-sheeted regular branched cyclic cover $p_k:S_k\to S_{1,2}$.

Ghaswala~\cite[Section 5.1.3]{ghaswala_thesis}, in his PhD thesis, derived a finite presentation for $\lmap_{p_k}(S_{1,2})$ when $k=2,3,4$ and a finite generating set when $k=5,6$ using Reidemeister-Schreier rewriting process on SageMath. We aim to derive a finite generating set for $\lmap_{p_k}(S_{1,2})$ for all $k\geq 2$. To achieve our aim, we use the homology representation $\Psi:\map(S_{1,2})\to \mathrm{GL}_3(\Z)$ of $\map(S_{1,2})$ induced by the action of $\map(S_{1,2})$ on $H_1(S_{1,2};\Z)$. Now, we state some results obtained by Ghaswala~\cite[Lemma 5.1.3 and 5.1.4]{ghaswala_thesis} describing $\Psi(\map(S_{1,2}))$ and $\Psi(\lmap_{p_k}(S_{1,2}))$. For a simple closed curve $c$, $T_c$ will denote the left-handed Dehn twist about $c$. Let $\iota\in \map(S_{1,2})$ be the hyperelliptic involution permuting the marked points on $S_{1,2}$ (see Figure~\ref{fig:s3_s12_cover}).

\begin{lem}
\label{lem:image_twists}
For the map $\Psi:\map(S_{1,2})\to \mathrm{GL}_3(\Z)$ and simple closed curves $a,b,c$ in Figure~\ref{fig:s3_s12_cover}, we have
\[
\Psi(T_a)=
\begin{pmatrix}
1 & 1 & 0 \\ 
0 & 1 & 0 \\ 
0 & 0 & 1
\end{pmatrix},
\Psi(T_b)=
\begin{pmatrix}
1 & 0 & 0 \\ 
-1 & 1 & 0 \\ 
0 & 0 & 1
\end{pmatrix},
\Psi(T_c)=
\begin{pmatrix}
1 & 1 & 0 \\ 
0 & 1 & 0 \\ 
0 & 1 & 1
\end{pmatrix},
\Psi(\iota)=
\begin{pmatrix}
-1 & 0 & 0 \\ 
0 & -1 & 0 \\ 
0 & 0 & -1
\end{pmatrix}.
\]
\end{lem}

\begin{proof}
The images of Dehn twists $T_a$, $T_b$, and $T_c$ was described in~\cite[Lemma 5.1.2]{ghaswala_thesis}. For the image of $\iota$, we observe that $\iota (x)=-x$ for $x=c_1,c_2,c_3$. Therefore, $\Psi(\iota)=-I_{3\times 3}$, where $I_{3\times 3}$ is the $3\times 3$ identity matrix.
\end{proof}

\begin{lem}
\label{lem:image_psi}
The image of $\map(S_{1,2})$ and $\lmap_{p_k}(S_{1,2})$ under the map $\Psi:\map(S_{1,2})\to \mathrm{GL}_3(\Z)$ are
\[
\Psi(\map(S_{1,2}))=
\left\{
\begin{pmatrix}
A & 0 \\ 
v & \pm1
\end{pmatrix}
\Big|~
A\in \mathrm{SL}_2(\Z), v=(m,n)\in \Z\times \Z
\right\}
\]
and
\[
\Psi(\lmap_{p_k}(S_{1,2}))=
\left\{
\begin{pmatrix}
A & 0 \\ 
v & \pm1
\end{pmatrix}
\Big|~
A\in \mathrm{SL}_2(\Z), v=(m,n)\in k\Z\times k\Z
\right\}.
\]
\end{lem}

As a corollary of Lemma~\ref{lem:image_psi}, it follows that the index of $\lmap_{p_k}(S_{1,2})$ in $\map(S_{1,2})$ is $k^2$ (see~\cite[Corollary 5.1.6]{ghaswala_thesis}). To derive a generating set for $\lmap_{p_k}(S_{1,2})$, we derive generating sets for $\ker \Psi$ and $\Psi(\lmap_{p_k}(S_{1,2}))$. A generating set for $\ker \Psi$ will be described in the next section. In the following result, we provide a finite generating set for $\Psi(\lmap_{p_k}(S_{1,2}))$. We will need the following elementary lemma.

\begin{lem}
\label{lem:gcd_sl2z}
For $m,n\in \Z$ such that $\ell=\gcd(m,n)$, there exists $A\in \mathrm{SL}_2(\Z)$ such that $(m,n)A=(0,\ell)$.
\end{lem}

\begin{proof}
Since $\ell=\gcd(m,n)$, there are $r,s\in \Z$ such that $mr+ns=\ell$. For
\[
A=\begin{pmatrix}
n/\ell & r \\ 
-m/\ell & s
\end{pmatrix}\in \mathrm{SL}_2(\Z)
\]
we have $(m,n)A=(0,\ell)$.
\end{proof}

\begin{prop}
\label{prop:genset_image}
For the map $\Psi:\map(S_{1,2})\to \mathrm{GL}_3(\Z)$ and curves shown in Figure~\ref{fig:s3_s12_cover}, we have $$\Psi(\lmap_{p_k}(S_{1,2}))=\langle \Psi(T_a),\Psi(T_b),\Psi(T_c^k),\Psi(\iota)\rangle.$$
\end{prop}

\begin{proof}
From Lemma~\ref{lem:image_psi}, it follows that $T_a,T_b,T_c^k,\iota \in \lmap_{p_k}(S_{1,2})$. Consider a matrix $X\in \Psi(\lmap_{p_k}(S_{1,2}))$. Up to a multiplication by $\Psi(\iota)$, we can assume that
\[
X=\begin{pmatrix}
A & 0 \\ 
v & 1
\end{pmatrix}, 
\]
where $A\in \mathrm{SL}_2(\Z)$ and $v=(m,n)\in k\Z\times k\Z$. By Lemma~\ref{lem:gcd_sl2z}, for $p=\gcd(m,n)$, there is a $B\in \mathrm{SL}_2(\Z)$ such that $(0,p)=(m,n)B$. We have
\[
\begin{pmatrix}
A & 0 \\ 
v & 1
\end{pmatrix}
\begin{pmatrix}
B & 0 \\ 
0 & 1
\end{pmatrix}=
\begin{pmatrix}
 AB & 0 \\ 
 u & 1
 \end{pmatrix}=:X_1 ,
\]
where $u=(0,p)$ such that $k\mid p$. Multiplying $X_1$ by $\Psi(T_c^k)^{-p/k}$ on the right, we get some matrix of the form
\[
X_2:=\begin{pmatrix}
C & 0 \\ 
0 & 1
\end{pmatrix} ,
\]
where $C\in \mathrm{SL}_2(\Z)$. Since $\mathrm{SL}_2(\Z)$ is generated by $\Psi(T_a)$ and $\Psi(T_b)$, it follows that there exists $Y\in \langle \Psi(T_a), \Psi(T_b), \Psi(T_c^k),\Psi(\iota)\rangle$ such that $XY=I_{3\times 3}$. Hence, the result follows.
\end{proof}

\section{Deriving a generating set for $\ker\Psi$}
\label{sec:genset_kerpsi}
Let $\Psi:\map(S_{1,2})\to \mathrm{GL}_3(\Z)$ be the homology representation for $\map(S_{1,2})$. In this section, we show that $\ker\Psi$ is a free group on countably many generators and it is normally generated by a separating Dehn twist. Furthermore, we provide an explicit free basis for $\ker\Psi$ consisting of separating Dehn twists. We will require the following observation.

\begin{lem}
\label{lem:exact_seqn_kernels}
Let $\psi:G \to H$ and $\phi:H  \to K$ be group homomorphisms and $\psi$ be surjective. We have the short exact sequence
\[
1 \longrightarrow \ker\psi \longrightarrow \ker\phi \circ \psi \longrightarrow \ker\phi\longrightarrow 1.
\]
\end{lem}

\begin{proof}
Define $\delta: \ker \phi \circ \psi\to \ker \phi$ as $\delta: x\mapsto \psi(x)$. Since $x\in \ker \phi \circ \psi$, $\phi \psi (x)=1$, that is, $\psi(x)\in \ker \phi$. This shows that $\delta$ is well-defined. First, we show that $\delta$ is surjective. Let $y\in \ker \phi$. Since $\psi$ is surjective, there is a $x\in G$ such that $\psi(x)=y$. Since $\phi \psi(x)=1$, we have that $x\in \ker \phi \circ \psi$. Thus, $\delta$ is surjective. Now, we show that $\ker \delta=\ker \psi$. We have $x\in \ker \delta$ if and only if $\delta (x)=\psi(x)=1$ if and only if $x\in \ker \psi$. Therefore, $\ker \delta =\ker \psi$. The result follows from the observation that $\ker \psi \subset \ker \phi \circ \psi$.
\end{proof}

We need the following result~\cite{tomaszewski}, which describes a free generating set for the commutator subgroup of a free group of rank $2$. The commutator $xyx^{-1}y^{-1}$ of $x,y$ in a group will be denoted by $[x,y]$.

\begin{lem}
\label{lem:genset_comm_free}
Let $F_2$ be a free group of rank $2$ with a free basis $\{x,y\}$. Then the commutator subgroup $[F_2,F_2]$ of $F_2$ is freely generated by the set $ \{x^my^n[x,y]y^{-n}x^{-m}\mid m,n\in \Z\}$.
\end{lem}

\noindent The following lemma describes some relations~\cite[Fact 3.9 and Proposition 3.11-3.12]{primer} between Dehn twists about simple closed curves.

\begin{lem}
\label{lem:relation_dehn_twist}
Let $c_1$ and $c_2$ be two simple closed curves on a surface $S$.
\begin{enumerate}[(i)]
\item If $i(c_1,c_2)=0$, then $[T_{c_1},T_{c_2}]=1$.
\item If $i(c_1,c_2)=1$, then the following relations (called braid relations) hold in $\map(S)$.
\begin{enumerate}[(a)]
\item $T_{c_1}T_{c_2}T_{c_1}=T_{c_2}T_{c_1}T_{c_2}$.
\item $T_{c_1}T_{c_2}T_{c_1}^{-1}=T_{c_2}^{-1}T_{c_1}T_{c_2}$.
\item $T_{c_1}T_{c_2}(c_1)=c_2$.
\end{enumerate} 
\end{enumerate} 
\end{lem}

Let $\pmap(S_{1,2})$ be the subgroup of $\map(S_{1,2})$ whose elements fix each marked points on $S_{1,2}$ point-wise. For $\psi:=\Psi|_{\pmap(S_{1,2})}$, observe that $\ker \Psi=\ker \psi$ and
\[
\mathrm{Im}(\psi)=
\left\{
\begin{pmatrix}
A & 0 \\ 
v & 1
\end{pmatrix}
\Big|~
A\in \mathrm{SL}_2(\Z), v=(m,n)\in \Z\times \Z
\right\}.
\]
Define the surjective homomorphism $\phi:\mathrm{Im}(\psi)\to \mathrm{SL}_2(\Z)$ as
\[
\phi:
\begin{pmatrix}
A & 0 \\ 
v & 1
\end{pmatrix}
\mapsto A \in \mathrm{SL}_2(\Z).
\]
It can be seen that
\[
\ker(\phi)=
\left\{
\begin{pmatrix}
1 & 0 & 0 \\ 
0 & 1 & 0 \\
m & n & 1 \\
\end{pmatrix}
\Big|~
m,n\in \Z
\right\}=\left\langle
\begin{pmatrix}
1 & 0 & 0 \\ 
0 & 1 & 0 \\ 
1 & 0 & 1
\end{pmatrix},
\begin{pmatrix}
1 & 0 & 0 \\ 
0 & 1 & 0 \\ 
0 & 1 & 1
\end{pmatrix} 
\right\rangle,
\]
that is, $\ker\phi\cong \Z \times \Z$. We apply Lemma~\ref{lem:exact_seqn_kernels} to maps $\psi$ and $\phi$ defined above to derive a generating set for $\ker\Psi$.

\begin{theorem}
\label{thm:genset_ker}
The kernel of $\Psi$ is isomorphic to the commutator subgroup $[F_2,F_2]$ of a free group $F_2$ of rank $2$ and it is normally generated by a separating Dehn twist. Furthermore, for curves as in Figure~\ref{fig:s3_s12_cover}, we have
\[
\ker\Psi= \langle T_b^{-n}T_cT_a^{-1}T_b^nT_a^m(T_bT_c)^6T_a^{-m}T_b^{-n}T_aT_c^{-1}T_b^n \mid m,n \in \Z \rangle.
\]
\end{theorem}

\begin{proof}
We recall the Birman short exact sequence~\cite[Theorem 4.6]{primer} induced by forgetting one of the marked points
\[
1 \longrightarrow \pi_1(S_{1,1})\xlongrightarrow{push}\pmap(S_{1,2})\xlongrightarrow{forget}\map(S_{1,1})\longrightarrow 1.
\]
Since $\map(S_{1,1})\cong \mathrm{SL}_2(\Z)$~\cite[Section 2.2.4]{primer}, we identify $\map(S_{1,1})$ with $\mathrm{SL}_2(\Z)$. We have $\phi\circ\psi:\pmap(S_{1,2})\to \mathrm{SL}_2(\Z)$. We claim that $\phi\circ \psi =forget$. It is known~\cite[Section 4.4.4]{primer} that $\pmap(S_{1,2})=\langle T_a,T_b,T_c \rangle$. The claim follows from the observation that
\[
\phi\circ \psi (T_a)=\phi\circ \psi (T_c)=
\begin{pmatrix}
1 & 1 \\ 
0 & 1
\end{pmatrix}=forget(T_a)=forget (T_c),\quad
\phi\circ \psi (T_b)=
\begin{pmatrix}
1 & 0 \\ 
-1 & 1
\end{pmatrix}=forget(T_b).
\]
By Lemma~\ref{lem:exact_seqn_kernels}, we have the short exact sequence
\[
1\to \ker\Psi \to \ker forget=\mathrm{Im}(push) \to \ker \phi \to 1.
\]
Since $\mathrm{Im}(push)\cong \pi_1(S_{1,1})\cong F_2$ and $\ker \phi \cong \Z \times \Z$, it follows that $\ker \Psi$ is the commutator subgroup $[F_2,F_2]$ of $F_2$, where $F_2$ is a free group of rank $2$. Let $\pi_1(S_{1,1})=\langle \bar{a},\bar{b}\rangle$ (see Figure \ref{fig:genset_pi_1_s12}). By Lemma~\ref{lem:genset_comm_free}, we have that $\ker \Psi$ is normally generated by $push([\bar{a},\bar{b}])$.
\begin{figure}[ht]
\centering
\includegraphics[scale=.7]{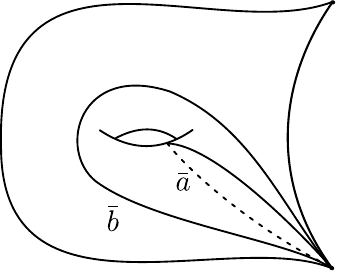}
\caption{A set of generating loops for $\pi_1(S_{1,1})$.}
\label{fig:genset_pi_1_s12}
\end{figure}

Now, we compute $push(\bar{a})$, $push(\bar{b})$, and $push([\bar{a},\bar{b}])$. For any $\gamma\in \pi_1(S_{1,1})$, we know that~\cite[Fact 4.7]{primer} $push(\gamma)=T_{\gamma_1}^{-1}T_{\gamma_2}$, where $\gamma_1$ and $\gamma_2$ are boundary curves of an annular neighborhood of $\gamma$. Therefore, $push(\bar{a})=T_c^{-1}T_a$, $push(\bar{b})=T_b^{-1}T_{\epsilon}$, and $push([\bar{a},\bar{b}])=T_d^{-1}T_{\eta}$ (see Figure~\ref{fig:push_images}).
\begin{figure}[ht]
\centering
\includegraphics[scale=0.75]{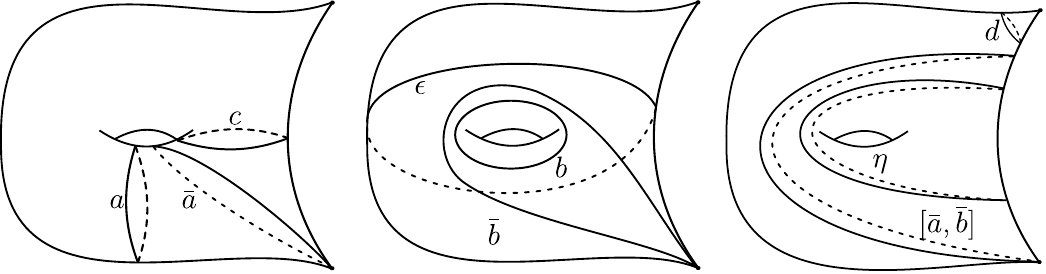}
\caption{Computation of $push(\bar{a})$, $push(\bar{b})$, and $push([\bar{a},\bar{b}])$.}
\label{fig:push_images}
\end{figure}
Since $\epsilon=T_cT_b(a)$, we have $T_{\epsilon}=T_cT_bT_aT_b^{-1}T_c^{-1}$. By the chain relation~\cite[Proposition 4.12]{primer} in $\map(S_{1,2})$, it follows that $T_{\eta}=(T_bT_c)^6$. Since $d$ is null-homotopic, we have $T_d=1$. Thus, $\ker \Psi$ is normally generated by a separating Dehn twist. By Lemma~\ref{lem:relation_dehn_twist}, we have $[T_c,T_{\eta}]=1=[T_c,T_a]$. It follows by Lemma~\ref{lem:genset_comm_free} that
\[
\ker\Psi= \langle T_b^{-n}T_cT_bT_a^nT_b^{-1}T_a^m(T_bT_c)^6T_a^{-m}T_bT_a^{-n}T_b^{-1}T_c^{-1}T_b^n \mid m,n \in \Z \rangle.
\]
By braid relations (Lemma~\ref{lem:relation_dehn_twist}), we have
\begin{align*}
&T_b^{-n}T_c(T_bT_a^nT_b^{-1})T_a^m(T_bT_c)^6T_a^{-m}(T_bT_a^{-n}T_b^{-1})T_c^{-1}T_b^n\\
&=T_b^{-n}T_cT_a^{-1}T_b^nT_a^{m+1}(T_bT_c)^6T_a^{-m-1}T_b^{-n}T_aT_c^{-1}T_b^n.
\end{align*}
Hence,
\[
\ker\Psi= \langle T_b^{-n}T_cT_a^{-1}T_b^nT_a^m(T_bT_c)^6T_a^{-m}T_b^{-n}T_aT_c^{-1}T_b^n \mid m,n \in \Z \rangle.
\]
\end{proof}

\section{A finite generating set for $\lmap_{p_k}(S_{1,2})$}
\label{sec:genset_lmod}
In this section, we provide a finite generating set for $\lmap_{p_k}(S_{1,2})$ by combining generating sets of $\mathrm{Im}(\Psi)$ from Proposition~\ref{prop:genset_image} and $\ker \Psi$ from Theorem~\ref{thm:genset_ker}.

\begin{theorem}
\label{thm:genset_lmod}
For $k\geq 2$, consider the $k$-sheeted cyclic cover $p_k:S_k \to S_{1,2}$. For curves as in Figure~\ref{fig:s3_s12_cover}, we have
\[
\lmap_{p_k}(S_{1,2})=\langle T_a,T_b,T_c^k, \iota,(T_bT_c)^6, T_c^jT_b^iT_a(T_bT_c)^6T_a^{-1}T_b^{-i}T_c^{-j}\mid 1\leq i,j <k \rangle. 
\]
\end{theorem}

\begin{proof}
By Proposition~\ref{prop:genset_image} and Theorem~\ref{thm:genset_ker}, it follows that $\lmap_{p_k}(S_{1,2})=\langle \{T_a,T_b,T_c^k,\iota\} \cup \s  \rangle$, where $\s=\{ T_b^{-n}T_cT_a^{-1}T_b^nT_a^m(T_bT_c)^6T_a^{-m}T_b^{-n}T_aT_c^{-1}T_b^n \mid m,n \in \Z\}$. Since $T_b, T_a$ lift under $p_k$ and $[T_c,T_a]=1$, the set $\s$ can be replaced by $\s_1=\{T_cT_b^nT_a^m(T_bT_c)^6T_a^{-m}T_b^{-n}T_c^{-1} \mid m,n \in \Z\}$. Since $[T_c,T_a]=1$, using braid relations (Lemma~\ref{lem:relation_dehn_twist}), we have 
\[
(T_cT_b^nT_c^{-1})T_a^m(T_bT_c)^6T_a^{-m}(T_cT_b^{-n}T_c^{-1})=T_b^{-1}T_c^nT_bT_a^{m}(T_bT_c)^6T_a^{-m}T_bT_c^{-n}T_b.
\]
Since $T_b$ and $T_c^k$ lift under $p_k$, the set $\s_1$ can be replaced by $\s_2=\{T_c^jT_bT_a^m(T_bT_c)^6T_a^{-m}T_b^{-1}T_c^{-j}\mid m\in \Z,~0\leq j<k\}$. For any $0\leq j<k$, by Lemma~\ref{lem:image_twists}, we can compute that
\[
\Psi(T_c^jT_b^kT_c^{-j})=
\begin{pmatrix}
1-kj & kj^2 & 0 \\ 
-k & 1+kj & 0 \\ 
-kj & kj^2 & 1
\end{pmatrix}.
\]
Thus, by Lemma~\ref{lem:image_psi}, $T_c^jT_b^kT_c^{-j}$ lifts under $p_k$. Let $m=kq+i$ for some $q,i\in \Z$ such that $0\leq i<k$. We have
\begin{align*}
&T_c^jT_bT_a^m(T_bT_c)^6T_a^{-m}T_b^{-1}T_c^{-j}\\
&=T_c^j(T_bT_a^mT_b^{-1})(T_bT_c)^6(T_bT_a^{-m}T_b^{-1})T_c^{-j}\\
&=T_a^{-1}T_c^jT_b^mT_a(T_bT_c)^6T_a^{-1}T_b^{-m}T_c^{-j}T_a\\
&=T_a^{-1}(T_c^jT_b^kT_c^{-j})^qT_c^{j}T_b^iT_a(T_bT_c)^6T_a^{-1}T_b^{-i}T_c^{-j}(T_c^jT_b^kT_c^{-j})^{-q}T_a.
\end{align*}
Thus, the set $\s_2$ can be replaced by $\s_3=\{ T_c^jT_b^kT_c^{-j}, T_c^{j}T_b^iT_a(T_bT_c)^6T_a^{-1}T_b^{-i}T_c^{-j} \mid 0\leq i,j<k \}$. For $\ell=\gcd(-kj,kj^2)$, by Lemma~\ref{lem:gcd_sl2z}, there is a matrix in $\mathrm{SL}_2(\Z)$ such that if we multiply it by $\Psi(T_c^jT_b^kT_c^{-j})$, we get a matrix of the form
\[
X=\begin{pmatrix}
A & 0 \\ 
v & 1
\end{pmatrix},
\]
where $A\in \mathrm{SL}_2(\Z)$ and $v=(0,\ell)$. Since $k\mid \ell$, we have $X\in \langle \Psi(T_a),\Psi(T_b),\Psi(T_c^k)\rangle$. Since $\ker\Psi\subset \lmap_{p_k}(S_{1,2})$, we can write $T_c^jT_b^kT_c^{-j}$ as a word in $T_a$, $T_b$, $T_c^k$, and elements form $\ker\Psi$. Now, it follows that $\s_3$ can be replaced with the set
\[
\s_4=\{(T_bT_c)^6, T_c^jT_b^iT_a(T_bT_c)^6T_a^{-1}T_b^{-i}T_c^{-j}\mid 1\leq i,j <k\}.
\]
Hence, the result follows. 
\end{proof}

For $k=2,3$, in the following result, we recover the generating sets obtained by Ghaswala~\cite[Section 5.1.3]{ghaswala_thesis}.

\begin{prop}
\label{prop:genset_lmod_23}
For $k=2,3$, we have $\lmap_{p_k}(S_{1,2})=\langle T_a,T_b,T_c^k,\iota\rangle$. 
\end{prop}

\begin{proof}
When $k=2$, from Theorem~\ref{thm:genset_lmod}, we have 
\[
\lmap_2(S_{1,2})=\langle T_a,T_b,T_c^2,\iota,(T_bT_c)^6,T_cT_bT_a(T_bT_c)^6T_a^{-1}T_b^{-1}T_c^{-1} \rangle.
\]
By Lemma~\ref{lem:relation_dehn_twist}, we have $T_cT_bT_a(b)=a$ and $T_cT_bT_a(c)=b$. Thus,
\begin{equation}
\label{eqn1}
T_cT_bT_a(T_bT_c)^6T_a^{-1}T_b^{-1}T_c^{-1}=(T_aT_b)^6.
\end{equation}
Furthermore, using braid relations, we have
\begin{equation}
\label{eqn2}
(T_bT_c)^6=(T_cT_b)^6=(T_c(T_bT_cT_b)T_cT_b)^2=  (T_c(T_cT_bT_c)T_cT_b)^2=(T_c^2T_b)^4.
\end{equation}
Now, it follows that $\lmap_2(S_{1,2})=\langle T_a,T_b,T_c^2,\iota \rangle$. 

When $k=3$, by Theorem~\ref{thm:genset_lmod}, we have $\lmap_3(S_{1,2})=\langle \s_1\cup \s_2\rangle$, where
$$\s_1=\{T_a,T_b,T_c^3,\iota,(T_bT_c)^6\}$$ and
\begin{align*}
\s_2=\{&T_cT_bT_a(T_bT_c)^6T_a^{-1}T_b^{-1}T_c^{-1},T_c^2T_bT_a(T_bT_c)^6T_a^{-1}T_b^{-1}T_c^{-2},\\
&T_cT_b^2T_a(T_bT_c)^6T_a^{-1}T_b^{-2}T_c^{-1},T_c^2T_b^2T_a(T_bT_c)^6T_a^{-1}T_b^{-2}T_c^{-2}\}
\end{align*}
Now, we reduce the generating set as described above. Using Equation~(\ref{eqn1}), we have
\begin{center}
$T_cT_bT_a(T_bT_c)^6T_a^{-1}T_b^{-1}T_c^{-1}=(T_aT_b)^6$ and $T_c^2T_bT_a(T_bT_c)^6T_a^{-1}T_b^{-1}T_c^{-2}=T_c(T_aT_b)^6T_c^{-1}$.
\end{center}
Furthermore, by using braid relations (Lemma~\ref{lem:relation_dehn_twist}), we have
\begin{align*}
T_cT_b^2T_a(T_bT_c)^6T_a^{-1}T_b^{-2}T_c^{-1}
&=(T_cT_b^2T_c^{-1})T_a(T_bT_c)^6T_a^{-1}(T_cT_b^{-2}T_c^{-1})\\
&=T_b^{-1}T_c^2T_bT_a(T_bT_c)^6T_a^{-1}T_b^{-1}T_c^{-2}T_b\\
&=T_b^{-1}T_c(T_aT_b)^6T_c^{-1}T_b
\end{align*}
Again, using braid relations (Lemma~\ref{lem:relation_dehn_twist}), we have
\begin{align*}
T_c^2T_b^2T_a(T_bT_c)^6T_a^{-1}T_b^{-2}T_c^{-2}
&=T_c^3(T_c^{-1}T_b^2T_c)T_a(T_bT_c)^6(T_c^3(T_c^{-1}T_b^2T_c)T_a)^{-1}\\
&=(T_c^3T_bT_c^3)T_c^{-1}T_b^{-1}T_a(T_bT_c)^6((T_c^3T_bT_c^3)T_c^{-1}T_b^{-1}T_a)^{-1}\\
&=(T_c^3T_bT_c^3)T_c^{-1}(T_b^{-1}T_aT_b)(T_bT_c)^6((T_c^3T_bT_c^3)T_c^{-1}(T_b^{-1}T_aT_b))^{-1}\\
&=(T_c^3T_bT_c^3T_a)T_c^{-1}T_bT_a^{-1}(T_bT_c)^6((T_c^3T_bT_c^3T_a)T_c^{-1}T_bT_a^{-1})^{-1}\\
&=(T_c^3T_bT_c^3T_a)T_c^{-1}T_b[(T_a^{-1}T_bT_a)T_c]^6((T_c^3T_bT_c^3T_a)T_c^{-1}T_b)^{-1}\\
&=(T_c^3T_bT_c^3T_a)T_c^{-1}T_b(T_bT_aT_b^{-1}T_c)^6((T_c^3T_bT_c^3T_a)T_c^{-1}T_b)^{-1}\\
&=(T_c^3T_bT_c^3T_a)T_c^{-1}T_b^2[T_a(T_b^{-1}T_cT_b)]^6((T_c^3T_bT_c^3T_a)T_c^{-1}T_b^2)^{-1}\\
&=(T_c^3T_bT_c^3T_a)T_c^{-1}T_b^2(T_aT_cT_bT_c^{-1})^6((T_c^3T_bT_c^3T_a)T_c^{-1}T_b^2)^{-1}\\
&=(T_c^3T_bT_c^3T_a)(T_c^{-1}T_b^2T_c)(T_aT_b)^6((T_c^3T_bT_c^3T_a)(T_c^{-1}T_b^2T_c))^{-1}\\
&=(T_c^3T_bT_c^3T_aT_bT_c^3)T_c^{-1}(T_aT_b)^6T_c(T_c^3T_bT_c^3T_aT_bT_c^3)^{-1}\\
\end{align*}
Hence, $\lmap_3(S_{1,2})=\langle T_a,T_b,T_c^3,\iota,(T_bT_c)^6,T_c(T_aT_b)^6T_c^{-1},T_c^{-1}(T_aT_b)^6T_c \rangle$. To prove the result, we are going to further reduce this generating set. Using braid relations (Lemma~\ref{lem:relation_dehn_twist}), we have
\begin{align*}
T_c(T_aT_b)^6T_c^{-1}
&=(T_aT_cT_bT_c^{-1})^6\\
&=(T_aT_b^{-1}T_cT_b)^6\\
&=T_b^{-1}(T_bT_aT_b^{-1}T_c)^6T_b\\
&=T_b^{-1}(T_a^{-1}T_bT_aT_c)^6T_b\\
&=T_b^{-1}T_a^{-1}(T_bT_c)^6T_aT_b.
\end{align*}
Similarly, we have
\begin{align*}
T_c^{-1}(T_aT_b)^6T_c
&=(T_aT_c^{-1}T_bT_c)^6\\
&=(T_aT_bT_cT_b^{-1})^6\\
&=T_b(T_b^{-1}T_aT_bT_c)^6T_b^{-1}\\
&=T_b(T_aT_bT_a^{-1}T_c)^6T_b^{-1}\\
&=T_bT_a(T_bT_c)^6T_a^{-1}T_b^{-1}.
\end{align*}
Finally, using Equation~(\ref{eqn2}), we have
\begin{align*}
(T_bT_c)^6
&=(T_cT_b)^6=(T_c^2T_b)^4\\
&=T_c^2T_bT_c^2T_bT_c^2T_bT_c^2T_b\\
&=T_c^3(T_c^{-1}T_bT_c)T_cT_bT_c^2T_bT_c^2T_b\\
&=T_c^3T_bT_c(T_b^{-1}T_cT_b)T_c^2T_bT_c^2T_b\\ &=T_c^3T_bT_c^2(T_bT_cT_b)T_c^2T_b\\
&=T_c^3T_bT_c^3T_bT_c^3T_b\\
&=(T_c^3T_b)^3.
\end{align*}
Therefore, we get $\lmap_3(S_{1,2})=\langle T_a,T_b,T_c^3,\iota\rangle$.
\end{proof}

Let $F_k\in \map(S_k)$ be the mapping class representing the $2\pi/k$-rotation of $S_k$ inducing the cover $p_k:S_k\to S_{1,2}$. There is a short exact sequence
\[
1\to \langle F_k\rangle\to N_{\map(S_k)}(F_k)\xrightarrow{\varphi}\lmap_{p_k}(S_{1,2})\to 1,  
\]
where $N_{\map(S_k)}(F_k)$ is the normalizer of $F_k$ in $\map(S_k)$. As a corollary of Proposition~\ref{prop:genset_lmod_23}, we provide a finite generating set of $N_{\map(S_k)}(F_k)$ for $k=2,3$ by lifting generating sets of $\lmap_{p_k}(S_{1,2})$ under the map $\varphi$.

\begin{cor}
\label{cor:genset_normalizer}
For curves shown in Figure~\ref{fig:s3_s12_cover}, we have
\[
N_{\map(S_2)}(F_2)=\langle F_2,T_{a_1}T_{a_2}, T_{b_1}T_{b_2},T_{c'},\hat{\iota} \rangle,
\]
where $\hat{\iota}\in \map(S_2)$ is the hyperelliptic involution, and
\[
N_{\map(S_3)}(F_3)=\langle F_3,T_{a_1}T_{a_2}T_{a_3}, T_{b_1}T_{b_2}T_{a_3},T_{c'},\iota_1\iota_2\iota_3 \rangle,
\]
where $\iota_1,\iota_2,\iota_3\in \map(S_3)$ are involutions represented by $\pi$-rotations as shown in Figure~\ref{fig:s3_s12_cover}.
\end{cor}

\begin{proof}
We observe that, under $p_2$, $a$ lifts to $\{a_1,a_2\}$, $b$ lifts to $\{b_1,b_2\}$, and $c$ lifts to $c'$. Furthermore, under $p_3$, $a$ lifts to $\{a_1,a_2,a_3\}$, $b$ lifts to $\{b_1,b_2,b_3\}$, and $c$ lifts to $c'$. The result follows from the fact that $\iota$ lifts to $\hat{\iota}\in \map(S_2)$ under $p_2$, and $\iota$ lifts to $\iota_1\iota_2\iota_3\in \map(S_3)$ under $p_3$. 
\end{proof}

\section{Maximality of $\lmap_{p_k}(S_{1,2})$}
\label{sec:maximal}
For $g\geq 2$ and the unbranched cover $p:S_{k(g-1)+1}\to S_g$, in~\cite[Theorem 2.4]{dhanwani_liftable2}, it was proved that $\lmap_p(S_g)$ is maximal in $\map(S_g)$ if and only if $k$ is prime. In this section, for the cover $p_k:S_k\to S_{1,2}$, we prove that $\lmap_{p_k}(S_{1,2})$ is maximal in $\map(S_{1,2})$ if and only if $k$ is prime. To prove this, we need the following lemma.

\begin{lem}
\label{lem:maximal}
Let $H$ be a subgroup of $G$ of index $n$. Let $\{g_1=1,g_2,\dots,g_n\}$ be a set of distinct left-coset representatives of $H$ in $G$. Then $H$ is maximal in $G$ if and only if $\langle H,g_i\rangle=G$ for every $2\leq i\leq n$.
\end{lem}

\begin{proof}
If $H$ is maximal in $G$, then it is apparent that $\langle H,g_i\rangle=G$ for every $2\leq i\leq n$. Conversely, assume that $\langle H,g_i\rangle=G$ for every $2\leq i\leq n$. Let $g\in G\setminus H$. Then there is some $j\neq 1$ such that $g\in g_jH$. Since $gH=g_jH$, we have $\langle H,g\rangle=\langle H,g_j\rangle=G$. Since $g\in G\setminus H$ was arbitrary, it follows that $H$ is maximal in $G$.
\end{proof}

\begin{theorem}
\label{thm:maximal}
Consider the $k$-sheeted cyclic cover $p_k:S_k\to S_{1,2}$. Then $\lmap_{p_k}(S_{1,2})$ is maximal in $\map(S_{1,2})$ if and only if $k$ is prime. 
\end{theorem}

\begin{proof}
Suppose $k$ is not prime and $\ell \mid k$ such that $1<\ell<k$. It is enough to show that $\Psi(\lmap_{p_k}(S_{1,2}))$ is not maximal in $\Psi(\map(S_{1,2}))$. For the subgroup
\[
K=
\left\{
\begin{pmatrix}
A & 0 \\ 
v & \pm1
\end{pmatrix}
\Big|~
A\in \mathrm{SL}_2(\Z), v=(m,n)\in \ell\Z\times \ell\Z
\right\},
\]
we have $\Psi(\lmap_{p_k}(S_{1,2}))\lneq K \lneq \Psi(\map(S_{1,2}))$ as $\Psi(T_c^\ell) \not\in \Psi(\lmap_{p_k}(S_{1,2}))$ and $\Psi(T_c)\not\in K$. Thus, $\lmap_{p_k}(S_{1,2})$ is not maximal in $\map(S_{1,2})$ when $k$ is not prime.

Now, assume that $k$ is prime. Let $\Psi_k:\map(S_{1,2})\to \mathrm{GL}_3(\Z_k)$ be the representation obtained from $\Psi$ by reduction modulo $k$. It is enough to show that $\Psi_k(\lmap_{p_k}(S_{1,2}))$ is maximal in $\Psi_k(\map(S_{1,2}))$. We have
\[
\Psi_k(\map(S_{1,2}))=
\left\{
\begin{pmatrix}
A & 0 \\ 
v & \pm1
\end{pmatrix}
\Big|~
A\in \mathrm{SL}_2(\Z_k), v=(m,n)\in \Z_k\times \Z_k
\right\}
\]
and
\[
\Psi_k(\lmap_{p_k}(S_{1,2}))=
\left\{
\begin{pmatrix}
A & 0 \\ 
0 & \pm1
\end{pmatrix}
\Big|~
A\in \mathrm{SL}_2(\Z_k)
\right\}.
\]
We observe that
\[
\left\{
\begin{pmatrix}
1 & 0 & 0 \\ 
0 & 1 & 0 \\
m & n & 1 \\
\end{pmatrix}
\Big|~
m,n\in \Z_k
\right\}=
\left\langle
\begin{pmatrix}
1 & 0 & 0 \\ 
0 & 1 & 0 \\ 
0 & 1 & 1
\end{pmatrix},
\begin{pmatrix}
1 & 0 & 0 \\ 
0 & 1 & 0 \\ 
-1 & 1 & 1
\end{pmatrix}
\right\rangle
\]
is the set of distinct coset representatives for $\Psi_k(\lmap_{p_k}(S_{1,2}))$ in $\Psi_k(\map(S_{1,2}))$. By Lemma~\ref{lem:image_twists}, it can be computed that
\[
\Psi_k(T_a^{-1}T_c)=
\begin{pmatrix}
1 & 0 & 0 \\ 
0 & 1 & 0 \\ 
0 & 1 & 1
\end{pmatrix} \text{ and }
\Psi_k(T_b^{-1}T_a^{-1}T_cT_b)=
\begin{pmatrix}
1 & 0 & 0 \\ 
0 & 1 & 0 \\ 
-1 & 1 & 1
\end{pmatrix}.
\]
Therefore, the set $\{\Psi_k[(T_b^{-1}T_a^{-1}T_cT_b)^m(T_a^{-1}T_c)^n]\mid m,n\in \Z_k\}$ is the set of distinct coset representatives for $\Psi_k(\lmap_{p_k}(S_{1,2}))$ in $\Psi_k(\map(S_{1,2}))$. We denote $H=\Psi_k(\lmap_{p_k}(S_{1,2}))$, $h_{m,n}=\Psi_k[(T_b^{-1}T_a^{-1}T_cT_b)^m(T_a^{-1}T_c)^n]$, and $H_{m,n}=\langle H,h_{m,n}\rangle$. Since $\Psi_k(T_a),\Psi_k(T_b)\in H$, we have $H_{m,n}=\langle H, \Psi_k(T_c^mT_bT_c^n)\rangle$. By Lemma~\ref{lem:image_twists}, we can compute that
\[
\Psi_k(T_c^mT_bT_c^n)=
\begin{pmatrix}
-m+1 & m-mn+n & 0 \\ 
-1 & -n+1 & 0 \\ 
-m & m-mn+n & 1
\end{pmatrix}.
\]
Now, assume that $(m,n)\neq (0,0)$, that is, $h_{m,n}\neq 1$. For $\gcd(-m,m-mn+n)=\gcd(n,-m)=\ell$, by Lemma~\ref{lem:gcd_sl2z}, there is a matrix in $\mathrm{SL}_2(\Z_k)$ such that if we multiply it with $\Psi_k(T_c^mT_bT_c^n)$, we get a matrix of the form
\[
X=\begin{pmatrix}
A & 0 \\ 
v & 1
\end{pmatrix}, 
\]
where $v=(0,\ell)$ and $A\in \mathrm{SL}_2(\Z_k)$. Since $\Psi_k(T_a), \Psi_k(T_b)\in H$, we can multiply $X$ with some matrix in $\mathrm{SL}_2(\Z_k)$ to get $\Psi_k(T_c^{\ell})\in \langle H, \Psi_k(T_c^mT_bT_c^n)\rangle=H_{m,n}$. Since $k$ is prime and $\ell=\gcd(n,-m)$, we have $\gcd(\ell,k)=1$. Therefore, there are integers $r,s$ such that $\ell r+ks=1$. Thus, $\Psi_k(T_c)=\Psi_k(T_c^{\ell})^r\Psi_k(T_c^k)^s$. Since $\Psi_k(T_c^k),\Psi_k(T_c^{\ell})\in H_{m,n}$, we get that $\Psi_k(T_c)\in H_{m,n}$ for $h_{m,n}\neq 1$. Thus $H_{m,n}=\Psi_k(\map(S_{1,2}))$ for $h_{m,n}\neq 1$. Hence, by Lemma~\ref{lem:maximal}, we get that $\lmap_{p_k}(S_{1,2})$ is maximal in $\map(S_{1,2})$ when $k$ is prime.
\end{proof}

\section{Acknowledgment}
The author like to thank Apeksha Sanghi for a careful reading of the paper and pointing out some typos in an earlier version of the paper.

\bibliographystyle{plain}
\bibliography{liftable_torus}
\end{document}